\documentclass[12pt,a4paper]{article}
\usepackage[english]{babel}
\usepackage{amsmath,amsthm}
\usepackage{amsfonts}
\usepackage{hyperref}
\usepackage[toc,page]{appendix}
\usepackage[stable]{footmisc}
\usepackage{amssymb}
\usepackage{resizegather}
\usepackage{graphicx}% http://ctan.org/pkg/graphicx
\usepackage{appendix}

\newtheorem{thm}{Theorem}

\newtheorem{lem}[thm]{Lemma}
\newtheorem{prop}[thm]{Proposition}
\theoremstyle{definition}
\newtheorem{defn}[thm]{Definition}
\theoremstyle{remark}

\newtheorem*{rem*}{\textbf{Remark}}

\numberwithin{subcase}{case}
\def\BAij{\mathbf{Bad}(i,j)}
\def\BAcij{\mathbf{Bad}_c(i,j)}
\def\BA{\mathbf{Bad}}
\def\BAi{\mathbf{Bad}(i)}
\def\C {\mathbf{C}}
\def\F {\mathbf{F}}
\def\R {\mathbb{R}}
\def\Z {\mathbb{Z}}
\def\N{\mathbb{N}}
\def\NEFES{\mathbb{N}\cup \{0\}}
\def\dl{\Delta(L)}

\def\jn{\mathcal{J}_n}
\def\j {\mathcal{J}}
\def\i {\mathcal{I}}
\def\m {\mathcal{M}}
\def\mn {\mathcal{M}_n}
\def\l {\mathcal{L}}
\def\t {\mathcal{T}}
\def\tn {\mathcal{T}_n}

\def\supp {{\rm supp}}

\begin{document}
\title{Badly approximable vectors on a vertical Cantor set}
\author{Erez Nesharim}
\date{March, 2013}
%\date{March, 2012}
%\thanks{Date: March, 2012}
\maketitle

\begin{abstract}
For $i, j > 0, i + j = 1$, the set of badly approximable vectors with weight $(i, j)$ is defined by $\BA(i, j) = \{(x, y) \in \R^2 : \exists c > 0 \;\forall q\in\N, \;\; \max\{q||qx||^{1/i} , q||qy||^{1/j} \} > c\}$, where $||x||$ is the distance from $x$ to the nearest integer. In 2010 Badziahin-Pollington-Velani solved Schmidt's conjecture which was stated in 1982, proving that $\BA(i, j) \cap \BA(j, i)$ is nonempty. Using Badziahin-Pollington-Velani's technique with reference to fractal sets, we were able to improve their results:
Assume that we are given a sequence $\left(i_t, j_t\right)$ with $i_t , j_t > 0, i_t + j_t = 1$. Then, the intersection of $\BA\left(i_t , j_t\right)$ over all t is nonempty.
\end{abstract}

\section{Introduction}\label{intro}
Let $i,j$ be such that
\begin{equation}\label{ij}
i,j\in[0,1],\;\;i+j=1.
\end{equation}
\begin{defn}[Badly approximable vectors with weights $(i,j)$]\label{baij}
%\small

\begin{equation}
\resizebox{.92\textwidth}{!}{$\displaystyle
\BAij=\left\{(x,y)\in\R^2:\exists c>0\;\;\forall p_1,p_2\in\Z,q\in\N \;\; \max\left\{q|qx-p_1|^\frac{1}{i},q|qy-p_2|^\frac{1}{j}>c\right\}\right\},
$}
\end{equation}
%\normalsize
and we agree that $\BA(1,0)=\mathbf{BA}\times\R$ and $\BA(0,1)=\R\times\mathbf{BA}$, where $\mathbf{BA}$ is the classical set of badly approximable numbers.
\end{defn}
\noindent
Schmidt's conjecture was concerned with the intersection between two different $\BAij$'s. It was proved by Badziahin-Pollington-Velani in \cite{Badziahin}. Actually, they proved
\begin{thm}\label{BPV1}
Let $\{\left(i_t,j_t\right)\}_{t\in\N}$ be as in (\ref{ij}). Assume
\begin{equation}\label{badCondition}
\liminf_t \min\{i_t,j_t\}>0.
\end{equation}
Then
$$\dim\left(\bigcap_{t=1}^\infty \BA\left(i_t,j_t\right)\right)=2.$$
\end{thm}
This solves Schmidt's conjecture about simultaneous Diophantine approximations. In fact, to prove this theorem, Badziahin-Pollington-Velani proved a theorem about the intersection of $\BAij$ with certain vertical intervals. To state it, first let us make the following definition:
\begin{defn}[Badly approximable numbers with weight $i$]\label{bai}
Let $0\leq i\in\R$. The set of badly approximable numbers with weight $i$ is
%\begin{equation}\label{baiEq}
$$\BAi=\left\{x\in\R:\exists c>0\;\forall p\in\Z,q\in\N \;\; q^\frac{1}{i}|qx-p|>c\right\},$$
%\end{equation}
where we agree on $\BA(0)=\R$.
\end{defn}
Notice that for any $i_1\leq i_2$, $\BA\left(i_2\right)\subseteq\BA\left(i_1\right)$, $\BA(1)=\mathbf{BA} $, and that for $i>1$, $\BAi=\varnothing$.
\begin{thm}[Badziahin-Pollington-Velani]\label{BPV2}
Let $\{\left(i_t,j_t\right)\}_{t\in\N}$ be as in (\ref{ij}). Denote $i=\sup_{t\in\N} i_t$ and assume (\ref{badCondition}). Assume
\begin{equation}\label{theta}
\theta\in\BAi,
\end{equation}
and let
\begin{equation}\label{Theta}
\Theta=\{(\theta,y):y\in[0,1]\}.
\end{equation}
Then,
\begin{equation}\label{5.5}
\dim\left(\bigcap_{t=1}^\infty \BA\left(i_t,j_t\right)\cap\Theta\right)=1.
\end{equation}
\end{thm}
In this paper we strengthen this result in two directions. The first direction is to consider the intersection of $\BAij$ with certain fractals. We will use a measure that is supported on the fractal. See \cite{KW}, \cite{KTV} for more on this subject, and \cite{icm2010} for a broader point of view.
\begin{defn}[Power Law]\label{powerLaw}
Let $X$ be a metric space, $\mu$ a Borel measure. $\mu$ satisfies a \emph{power law} if there are positive $\beta,b_1,b_2$ such that $\forall x\in supp(\mu),0<r<1$,
\begin{equation}\label{powerLawEq}
b_1r^\beta \leq \mu(B(x,r)) \leq b_2r^\beta.
\end{equation}
\end{defn}
\noindent
Using this property we prove
\begin{thm}\label{dim}
Let $i,j\in[0,1]$ be as in (\ref{ij}), $\theta$ as in (\ref{theta}) and $\Theta$ be as in (\ref{Theta}). Assume $\C\subseteq\Theta$ is the support of a probability measure $\mu$ on $\Theta$, which satisfies a power law with exponent $\beta$. Then for any $\beta'<\beta$, there exists a measure $\nu$ satisfying a power law with exponent bigger than $\beta'$, such that
$$\supp(\nu)\subseteq\BAij\cap\C.$$
In particular,
$$\dim(\BAij\cap\C)=\beta.$$
\end{thm}
\noindent
This result with $\C=\Theta$ is the case of a single $\BAij$ in Theorem~\ref{BPV2}. Badziahin-Pollington-Velani asked whether (\ref{5.5}) is true without assuming (\ref{badCondition}). Our second strengthening of \cite{Badziahin} provides a partial result to this question.
\begin{thm}\label{ourThm}
Let $\C\subseteq\Theta$ be the support of a measure satisfying a power law, and let $\{\left(i_t,j_t\right)\}_{t\in \N}$ with $\left(i_t,j_t\right)$ as in (\ref{ij}). Then
$$\C\cap\bigcap_{t\in\N} \BA\left(i_t,j_t\right)\neq\varnothing.$$
\end{thm}
\noindent
Using the techniques of this article one cannot give a result about the dimension of the infinite intersection. Recently, Jinpeng An\cite{JinpengAn} proved that in the case $\C=\Theta$,
$$\BAij\cap\C\neq\varnothing\Rightarrow\;\BAij\cap\C {\rm \; is \; winning},$$
which in particular implies that any countable intersection of such sets is not empty. In \ref{abswin}, which is joint with Barak Weiss, we use Jinpeng An's result and method in order to prove
$$\BAij\cap\C\neq\varnothing\Rightarrow\;\BAij\cap\C {\rm \; is \; winning},$$
which easily gives also a dimension result in the context of Theorem~\ref{ourThm}.

\vspace{8 mm}
The structure of this paper is the following. In Section~\ref{mainTheorem} we prove Theorem~\ref{dim} assuming Theorem~\ref{main_thm} which will be stated there. The proof uses the method developed in \cite{Badziahin}, and some propositions from that paper are used without proof. In section~\ref{conclusions} we prove Theorem~\ref{ourThm}. In Section~\ref{mainKey} we prove the crucial Theorem~\ref{main_thm} that is used in Section~\ref{mainTheorem}.

\vspace{8 mm}
\textbf{acknowledgements:} I would like to thank my advisor Barak Weiss for many helpful and encouraging discussions, as well as many suggestions during this work. I am very grateful to the referee for carefully reading this paper and spotting some inaccuracies in previous versions. Also, I thank the editor for some comments about the typing of this paper.

\section{Main Theorem}\label{mainTheorem}
Before we give the proof of Theorem~\ref{dim}, we need some notations and lemmata. For any $c>0$
%\begin{equation}\label{bacij}
define
\begin{equation}
\resizebox{.92\textwidth}{!}{$\displaystyle
\BAcij = \left\{(x,y)\in\R^2:\forall A,B,C\in\Z,(A,B)\neq (0,0)\;\; \max\left\{|A|^\frac{1}{i},|B|^\frac{1}{j}\right\}|Ax+By+C|>c\right\}.
$}
\end{equation}
We remark that here we use the dual formulation for $\BAcij$. By using a transference principle (cf. e.g. \cite{Badziahin}, Appendix), we note that
$$\BAij = \bigcup_{c>0}\BAcij.$$
Viewing it in this form, we see that (\ref{theta}) is a necessary condition on $\theta$ for the existence of a $y\in\R$ such that $(\theta,y)\in\BAij$.
For any $\C\subseteq\Theta$
\begin{equation}
\resizebox{.92\textwidth}{!}{$\displaystyle
\BAcij\cap\C=\C\backslash\bigcup_{(A,B,C)\in \Z^3\backslash\{0\}} \left\{(x,y):|Ax-By+C|\leq\frac{c}{\max\left\{|A|^\frac{1}{i},|B|^\frac{1}{j}\right\}}\right\}.
$}
\end{equation}
\noindent
For $B\neq0$, we see that a line
$$L(A,B,C):Ax-By+C=0$$
intersects $\Theta$ at a point $(\theta,y(L))$ where
$$y(L)=\frac{A\theta+C}{B}.$$
Denote by $\dl$ the points $(\theta,y)\in\Theta$ that are not in $\BAcij$ because they are too close to $(\theta,y(L))$, that is
%\begin{equation}\label{dl}
$$\dl = \Theta \cap \left\{(x,y):|Ax-By+C|\leq\frac{c}{\max\left\{A^\frac{1}{i},B^\frac{1}{j}\right\}}\right\}.$$
%\end{equation}
Dividing by $B$ we get
\begin{equation}\label{dlLength}
|\dl|=\frac{2c}{H(A,B)},
\end{equation}
where if $I$ is an interval then $|I|$ is the diameter of $I$, and
%\begin{equation}\label{H}
$$H(A,B) \stackrel{\mathrm{def}}{=} B\cdot\max\left\{|A|^\frac{1}{i},|B|^\frac{1}{j}\right\}.$$
%\end{equation}
The plan is to prove that after removing all intervals $\dl$, still most of $\C$ is not removed.
We do it by constructing (recursively) a sequence of collections of disjoint intervals $\{\jn\}_{n\in\NEFES}$, for which
$$\forall n\in\N,J\in\jn\;\;\exists J'\in\j_{n-1}$$
such that
$$J=B\left(y_J,r\right)=\{y\in\R:d\left(y,y_J\right)\leq r\},$$
where $r=\frac{1}{2}c_1R^{-n}$ ($c_1$ is defined below in (\ref{c1})), $y_J\in J'$ and $J$ satisfies
\begin{equation}\label{inhyp1}
\dl\cap J=\varnothing \;\; {\rm for \; every} \;\; L=L(A,B,C) \;\; {\rm with} \;\; H(A,B)<R^{n-1},
\end{equation}
\noindent
and $R=R\left(i,j,b_1,b_2,\beta,\beta'\right)$ is a fixed integer that we choose later (cf. (\ref{Rtotal})). $\theta\in\BAi$ so by definition, there exists $c(\theta)$ that fulfils
$$\forall p\in\Z,q\in\N \;\; q^\frac{1}{i}|qx-p|>c(\theta).$$
Note that for any $c\leq c(\theta)$ it is enough to consider only lines $L(A,B,C)$ with
\begin{equation}
\label{line1}
\gcd(A,B,C)=1,\;\; B > 0.
\end{equation}
\noindent
This is the place to note that in the case $i=1,j=0$ we have $\BAij\cap\Theta=\Theta$, and the assertion of the theorem is classical. In the other extreme, $i=0,j=1$ we actually have $\BAij\cap\Theta=\{\theta\}\times(\mathbf{BA}\cap[0,1])$. Although we could modify the construction to deal with this case (cf. \cite{Badziahin}, Chap. 3.2), we note that the assertion of the theorem in this case is already known, proved independently in \cite{KW} and \cite{KTV}. We proceed assuming $i,j\neq 0$. Let
\begin{equation}\label{c1}
c_1 = \min\left\{c(\theta)R^{1+\alpha},\frac{1}{4}R^{-\frac{3i}{j}}\right\},
\end{equation}
where
\begin{equation}\label{alpha}
\alpha=\frac{\beta ij}{4}.
\end{equation}
Then,
\begin{equation}\label{c}
c=\frac{c_1}{R^{1+\alpha}}\leq c(\theta).
\end{equation}
We start the construction by looking at the following collection of closed subintervals of $\Theta$,
$$\tilde{\i}_0=\left\{B\left(y,\frac{1}{2}c_1\right):(\theta,y)\in\supp(\mu)\right\}.$$
By the $5r$-covering lemma (\cite{Mattila}, Chap. 2), choose a set of disjoint subintervals $\i_0\subseteq\tilde{\i}_0$ such that
$$\bigcup_{I\in\tilde{\i}_0}I\subseteq\bigcup_{I\in\i_0}5I,$$
where if $I=B(y,r)$, $\gamma\geq0$ then $\gamma I=B(y,\gamma r)$. In particular $\mu(\bigcup_{I\in\i_0}5I)=\mu(\Theta)=1$, since $\mu$ is a probability measure. For every $I\in\i_0$, $|I|=c_1$. Using the right hand side of (\ref{powerLawEq}) we get $\mu(5I)\leq b_2\left(\frac{5}{2}c_1\right)^\beta$ and
$$\#\i_0\geq \frac{\mu(\Theta)}{\max_{I\in\i_0}\mu(5I)}\geq b_2^{-1}\left(\frac{5}{2}c_1\right)^{-\beta},$$
where $\#$ denotes the number of elements of a finite set. Set $\j_0=\i_0$. This finishes the construction of the zero'th level. Let $n\in\N$ and assume that we are given the collections $\i_n,\jn$  and that $\jn$ satisfies (\ref{inhyp1}). Denote the collection of lines we should avoid in the $(n+1)$'th step by
\begin{equation}\label{Cn}
C(n)=\{L(A,B,C):L \;\; {\rm satisfies} \;\; (\ref{line1}) \;\; {\rm and} \;\; (\ref{H1}) \}
\end{equation}
where
\begin{equation}\label{H1}
R^{n-1}\leq H(A,B)<R^n.
\end{equation}
Notice that, using (\ref{dlLength}) and the definition of $c$ in (\ref{c}), a line $L\in C(n)$ satisfies
$$|\dl|=\frac{2c}{H(A,B)}\leq 2cR^{-n+1} \leq 2c_1R^{-n-\alpha}.$$
For each $I\in\i_n$ define the subinterval
$$I^- = \left(1-R^{-\alpha}\right)I.$$
The motivation for that is to ensure that every two disjoint intervals $I_1,I_2\in\i_n$ and a line $L\in C(n)$ satisfy
$$\Delta(L)\cap I_1^{-}\neq\varnothing\;\;\Rightarrow\;\;\Delta(L)\cap I_2^{-}=\varnothing.$$
and that for every $I\in\i_n$,
\begin{equation}\label{tmp}
2\Delta(L)\cap I^{-}\neq\varnothing\;\;\Rightarrow\;\;\Delta(L)\cap I\neq\varnothing.
\end{equation}
Next, for every $I'\in\i_n$ we define the intermediate collection
$$\tilde{\i}_{n+1}(I')=\left\{B\left(y,\frac{1}{2}c_1R^{-n-1}\right):(\theta,y)\in\supp(\mu)\cap {I'}^-\right\},$$
Apply the $5r$-covering lemma to $\tilde{\i}_{n+1}(I')$ to get a disjoint collection of subintervals $\i_{n+1}(I')$ such that
\begin{equation}\label{inplus1j}
\bigcup_{I\in\tilde{\i}_{n+1}(I')}I \subseteq \bigcup_{I\in\i_{n+1}(I')}5I.
\end{equation}
Define
\begin{equation}\label{inplus0}
\i_{n+1}=\bigcup_{I'\in\i_n}\i_{n+1}(I'),
\end{equation}
\begin{equation}\label{inplus1}
\i_{n+1}(\j)=\bigcup_{J\in\jn}\i_{n+1}(J).
\end{equation}
Note that
$|(I')^-| = c_1R^{-n}\left(1-R^{-\alpha}\right),$
and by (\ref{powerLaw}), for every $I\in\i_{n+1}(I')$, $\mu(5I)\leq b_2\left(\frac{5}{2}c_1R^{-(n+1)}\right)^\beta$ so
\begin{equation}
\resizebox{.9\textwidth}{!}{$\displaystyle
\#\i_{n+1}(I')\geq \frac{\mu({I'}^-)}{\max_{I\in\i_{n+1}}\mu(5I)}\geq \frac{b_1}{b_2}\left(\frac{|{I'}^-|}{|5I|}\right)^\beta=\frac{b_1}{5^\beta b_2}\left(R\left(1-R^{-\alpha}\right)\right)^\beta.
$}
\end{equation}
For the ease of calculations, take $R$ such that $R^{-\alpha}\leq \frac{1}{2}$ and $\beta\leq1$ so
\begin{equation}\label{jnNum}
\#\i_{n+1}(I')\geq \frac{b_1}{10b_2}R^\beta.
\end{equation}
To define $\j_{n+1}$, we remove intervals $I\in\i_{n+1}(\j)$ that intersect some $\Delta(L)$ for a line $L\in C(n)$, that is
$$\j_{n+1}=\{I\in\i_{n+1}(\j): \forall L\in C(n) \;\; \dl\cap I=\varnothing\}.$$
\noindent
We must show that $\j_{n+1}\neq\varnothing$, but in order to construct a measure with its support in $\C$ it is not enough to have an estimate on $\#\jn$. Rather, it is necessary to know more about the structure of $\{\j_n\}_{n\in\NEFES}$. Namely, we wish to use the notion of a tree-like family as in \cite{KW}. Unfortunately, $\{\j_n\}$ might have finite branches and we must pass to a subcollection.  Following \cite{Badziahin}, define,
\begin{equation}\label{Cnl}
C(n,\ell) \stackrel{\mathrm{def}}{=} \left\{L\in C(n): R^{-\lambda(\ell+1)}R^{\frac{nj}{j+1}}\leq B<R^{-\lambda \ell}R^{\frac{nj}{j+1}}\right\}, \;\; n,\ell \geq 0
\end{equation}
where
\begin{equation}\label{lambda}
\lambda=\frac{3}{j}.
\end{equation}
Recall that for $L(A,B,C)\in C(n)$, $B\geq1$ since (\ref{line1}) is satisfied, and
$$R^n>H(A,B)=B\max\left\{A^\frac{1}{i},B^\frac{1}{j}\right\}\geq B^{\frac{1+j}{j}},$$
so $B<R^{\frac{nj}{j+1}}$. Therefore, $C(n,\ell)$ is empty for $\ell>\frac{nj}{\lambda(j+1)}$ and for $\ell<0$, so
$$\bigcup_{\ell=0}^{\frac{nj}{\lambda(j+1)}}C(n,\ell)=C(n).$$
The following theorem is most important for our proof and Section~\ref{mainKey} is devoted to it.
\begin{thm}\label{main_thm}
Let $n,\ell\geq 0$, $\ell\leq\frac{nj}{\lambda(j+1)}$, and $J\in\j_{n-\ell}$. Let
\begin{equation}\label{epsilon0}
\varepsilon = \frac{\alpha\beta^2 ij}{20},
\end{equation}
and $R\geq R_1$ where
\begin{equation}\label{R1}
R_1 = \max\left\{R_0,\left(\frac{64b_2^2}{b_1^2}\right)^{\frac{10}{\alpha\beta^2 ij}},c_5^{\frac{2}{\alpha\beta}}\right\},
\end{equation}
$R_0$ is the solution of the equation
\begin{equation}\label{R0}
R_0^{\varepsilon}=\log_2R_0,
\end{equation}
and $c_5$ is as in (\ref{c5}).
Then,
\begin{equation}\label{main_thm_eq}
\#\{I\in\i_{n+1}(J):\exists L\in C(n,\ell) \;\; I\cap\dl\neq\varnothing\}\leq R^{\beta-\varepsilon}.
\end{equation}
where $\i_{n+1}(J) = \{I\in\i_{n+1}:I\subseteq J\}$ (For $J\in\jn$ this definition for $\i_{n+1}(J)$ coincides with the definition in (\ref{inplus1j})).
\end{thm}
\noindent
Informally speaking, Theorem~\ref{main_thm} says that our family $\j_n$ is a tree, for which every father has more than $\frac{b_1}{10b_2}R^\beta$ children (cf. (\ref{jnNum})), minus $R^{\beta-\varepsilon}$ vertices that may be removed by every father from every generation that descends it. (more precisely, a father in the $n_0$'th generation, is able to remove children from the $n$'th generation whenever $n>n_0$ satisfies $n-\frac{nj}{\lambda(j+1)}\leq n_0,$ that is $n\leq\frac{\lambda(j+1)}{\lambda(j+1)-j}n_0$.) In this situation it may be the case that some $J\in\j_n$ doesn't contain even a single element from $\j_{n+1}$. Nevertheless, there exists a subcollection on which the number of children is bounded from below.
\begin{defn}\label{treeLikeDef}
A tree-like family of intervals is a union of collections of closed intervals $\t = \bigcup_{n\in\NEFES}\t_n$ such that $\t_0=\{J_0\}$ and it satisfies the following:
\begin{enumerate}
\item $\forall I\in\t\;\;|I|>0$.
\item $\forall n\in\N \;\forall I_1,I_2\in\t_n\;\;{\rm either}\;\; I_1=I_2 \;\; {\rm or}\;\; \#I_1\cap I_2\leq1$.
\item $\forall n\in\N \;\forall I\in\t_n \;\exists J\in\t_{n-1} \;\;I\subseteq J$.
\item $\forall n\in\N \;\forall J\in\t_{n-1}\;\;\t_n(J)\neq\varnothing$, where
$$\t_n(J)=\{I\in\t_n:I\subseteq J\}.$$
\end{enumerate}
For $r\in\N$, the tree-like family is called \emph{$r$-regular} or \emph{regular of degree $r$} if for every $n\in\N,J\in\t_{n-1}$ $$\#\t_n(J)=r.$$
\end{defn}
The following property is proved in (\cite{Badziahin}, Chap.$7$, Lemma 4). We present the proof again to extend its context to ours.
\begin{lem}[`Ubiquity' of $\jn$]\label{ubique}
Let $J_0\in\j_0$, $\varepsilon$ as in (\ref{epsilon0}), $R\geq \max\{R_1,R_2\}$ where $R_1$ is as in (\ref{R1}), and
\begin{equation}\label{R2}
R_2=2^{\frac{2}{\beta}}.
\end{equation}
Let $\t$ be a regular tree-like subfamily of $\i=\bigcup_{n\in\NEFES}\i_n$ of degree $\lceil3R^{\beta-\varepsilon}\rceil,$ with $\t_0=\{J_0\}$.
Then, $\forall n\in\N$
$$\t_{n} \cap \j_{n} \neq \varnothing.$$
\end{lem}
\begin{proof}[Proof of Lemma~\ref{ubique} using Theorem~\ref{main_thm}]
Define the sequence
$$f(n)=\#\left(\jn\cap\tn\right),\;\; n\in\NEFES.$$
Using induction we will show that for every $n\in\NEFES$,
$$f(n)\geq R^{\beta-\varepsilon}f(n-1).$$
Assume $n\in\NEFES$. We will bound from above the number of intervals from $\t_{n+1}$ that aren't in $\j_{n+1}$. By (\ref{main_thm_eq}) we know that for each $1 \leq\ell\leq \frac{(n+1)j}{\lambda(j+1)}$, each father from $\ell$ generations above can remove no more than $R^{\beta-\varepsilon}$ intervals from each level of its successor. Considering the fact that only fathers from our $\t$ participate in that, the number of intervals that may be removed in this way is less than
$$\sum_{\ell=1}^{\frac{(n+1)j}{\lambda(j+1)}}R^{\beta-\varepsilon}f(n+1-\ell).$$
Repeatedly using the induction hypothesis up to $n$, we have
$$f(n-\ell)\leq \left(R^{\beta-\varepsilon}\right)^{-\ell}f(n).$$
Using (\ref{R2}) and (\ref{epsilon0}) we get $R^{\varepsilon-\beta}\leq\frac{1}{2}$ so
$$\sum_{\ell=0}^{\infty} R^{(\varepsilon-\beta)\ell} \leq 2.$$
Finally,
\begin{align*}
f(n+1)
&\geq
\lceil3R^{\beta-\varepsilon}\rceil f(n) - \sum_{\ell=1}^{\frac{(n+1)j}{\lambda(j+1)}}R^{\beta-\varepsilon}f(n+1-\ell)
\\ &\geq
3R^{\beta-\varepsilon}f(n) - R^{\beta-\varepsilon}f(n)\sum_{\ell=0}^\infty R^{(\varepsilon-\beta)\ell}
\geq
R^{\beta-\varepsilon}f(n).
\end{align*}
In particular $f(n)>0$ and we are done.
\end{proof}
\begin{defn}\label{ubiquity}
Let $F$ be a tree and assume $T\subseteq F$ is a subtree. For $r\in\N$, $T$ is said to have \emph{r-ubiquity with respect to $F$} if every regular tree of degree $r$, $F_r\subseteq F$, satisfies
$$F_r(n) \cap T(n) \neq \varnothing, \;\; \forall n\in\NEFES,$$
where $F_r(n)$ and $T(n)$ stands for the sets of vertices in the $n$'th generation of the tree.
\end{defn}
\noindent
Inspired by subsection $7.3$ in \cite{Badziahin}, we prove the following
\begin{thm}\label{rubiquity}
Assume $r_0\in\N$, $F_{r_0}$ is a regular tree of degree $r_0$, and $T\subseteq F_{r_0}$ is a tree with $r$-ubiquity with respect to $F_{r_0}$. Then there exists a regular tree of degree $r_0-r+1$ that is contained in $T$.
\end{thm}
\begin{proof}
It is enough to prove the existence of a finite tree of any length. Indeed, assume we have a collection of regular subtrees of degree $r_0-r+1$ of every length, $\{T_n\}_{n\in\N}$. Generate an infinite tree $T_{\infty}$ by choosing the first generation of it to be $r_0-r+1$ vertices that appear infinitely many times in the finite trees $T_n$. Continue by induction, and choose the $m$'th level of $T_{\infty}$ to be vertices that appear infinitely many times in the trees $\{T_n\}_{n\geq m}$ that have the same $m-1$ level as $T_\infty$.

To prove existence of a tree of any finite length, we argue by induction on the length. For a tree of length $0$ the assertion is empty. Assume that every tree of length $n$ with $r$-ubiquity contains a regular subtree of degree $r_0-r+1$, and look at our tree $T$ up to level $n+1$. For at least $r_0-r+1$ vertices of the first generation, $v\in T(1)$, the tree $T^v$, which starts in $v$ and contains every vertex of $T$ that have $v$ as its ancestor, has $r$-ubiquity. Otherwise, construct a regular tree of degree $r$ that contradicts $r$-ubiquity by choosing its first level to be $r$ vertices for which $T^v$ doesn't have $r$-ubiquity. For every such $v$ there exists a regular subtree $F_{r,v}$ and $n_v\in\N$ such that $T^v\left(n_v\right)\cap F_{r,v}\left(n_v\right)=\varnothing$. This defines a tree $F_r$ for which we have
$$T^v(n)\cap F_{r}(n)=\varnothing,$$
where $n=\max_{v\in T(1)}\{n_v\}$, and therefore contradicts $r$-ubiquity of $T$.
Now, to construct a regular tree we can choose $r_0-r+1$ vertices $v$ from $T(1)$ for which $T^v$ has $r$-ubiquity. Use the induction hypothesis to find a regular tree of degree $r_0-r+1$ in each $T^v$ and use it to continue our regular tree up to level $n+1$. Thus we have found a regular tree of degree $r_0-r+1$ and of length $n+1$ which is contained in $T$.
\end{proof}

\begin{proof}[Deduction of Theorem \ref{dim} from Lemma~\ref{ubique} and Theorem~\ref{rubiquity}]
Let $\varepsilon$ be as in (\ref{epsilon0}), let $R_1,R_2$ be as in (\ref{R1}) and (\ref{R2}). Let
\begin{equation}\label{Rtotal}
R\geq \max\{R_1,R_2,R_3\},
\end{equation}
where
$R_3=\left(\frac{60b_2}{b_1}\right)^{\frac{1}{\varepsilon}}$.
Now take any regular subtree $\i'$ of $\i$ with degree $r_0=\lceil\frac{b_1}{10b_2}R^\beta\rceil$. There exists such a subtree because of (\ref{jnNum}). It is clear from Lemma~\ref{ubique} that the family $\{\j_n\}_{n\in\NEFES}$ has $r$-ubiquity with respect to $\i'$, with $r=\lceil3R^{\beta-\varepsilon}\rceil$. By Theorem~\ref{rubiquity} we can choose a collection $\tilde{\m}_n~\subseteq~\j_n$ such that for every $J'\in\tilde{\m}_n$,
\begin{equation}\label{numMtilde}
\# \{J\in\tilde{\m}_{n+1}(J')\} = \lceil\frac{b_1}{10b_2}R^\beta\rceil-\lceil3R^{\beta-\varepsilon}\rceil+1 \geq \lceil\frac{b_1}{20b_2}R^\beta\rceil,
\end{equation}
where the last inequality is true because $R\geq R_3$. Let $\{\mn\}_{n\in\NEFES}$ be such that $\mn\subseteq\tilde{\m}_n$ for every $n\in\N$ and equality holds in (\ref{numMtilde}), i.e.,
$$\# \{J\in\m_{n+1}(J')\} = \lceil\frac{b_1}{20b_2}R^\beta\rceil.$$
Note that we use $\m_0=\j_0$, but for calculating dimension we can ignore any finite number of levels of the construction. Denote
$$K_c = \bigcap_{n\in\NEFES}\bigcup_{J\in\mn}J.$$
To define the measure we want on $K_c$ we use the following standard lemma, proved in \ref{treeLikeAppendix}:
\begin{lem}\label{treeDim}
Let $\{\t_n\}_{n\in\NEFES}$ be a tree-like family of intervals. Assume that there exists $n_0\in\NEFES$ and $\gamma,R>0$ such that $\forall n\geq n_0$, $J\in\t_n$
$$\forall I\in\t_{n+1}(J)\;\;|I|=\frac{|J|}{R},$$
\begin{equation}\label{numOfChildren}
\#\t_{n+1}(J)=\gamma R.
\end{equation}
Then there exists a measure $\nu$ with $\supp(\nu)=\bigcap_{n\in\NEFES}\bigcup_{I\in\t_n}I$ satisfying a power law with exponent $\beta = \log_R(\gamma R)$.
\end{lem}
\noindent
$\{\mn\}_{n\in\NEFES}$ satisfies the conditions of Lemma~\ref{treeDim} with
$\gamma = \frac{\lceil\frac{b_1}{20b_2}R^\beta\rceil}{R}$ and $n_0=1$. Therefore for every $R$ as in (\ref{Rtotal}) and $c=c(R)$ as in (\ref{c}) there exists a measure $\mu_c$ on $K_c$ satisfying a power law with an exponent
%\begin{equation}\label{betac}
$$\beta_c = \log_R(\gamma R) = \beta-\log_R\frac{R^\beta}{\lceil\frac{b_1}{20b_2}R^\beta\rceil}\geq\beta-\log_R\frac{20b_2}{b_1}.$$
%\end{equation}
$\lim_{R\rightarrow\infty}\beta_{c(R)}=\beta$ so we have proved the main part of Theorem~\ref{dim}. $K_c\subseteq\BAij\cap\C$ so using the easy part of Frostman's lemma (\cite{Mattila}, Chap. 8), we get
$\dim(\BAij\cap\C)\geq\beta_{c(R)}$ for every $R$ as in (\ref{Rtotal}), so $\dim(\BAij\cap\C)=\beta$.
\end{proof}

\section{Conclusions}\label{conclusions}
In proving Theorem~\ref{ourThm} we need to be a little bit careful because of the fact that the sets $\BAij$ are not closed. Instead, we work with the support of the measure constructed in Theorem~\ref{dim}.
\begin{proof}[proof of Theorem~\ref{ourThm}]
Let $\varepsilon>0$. Use Theorem~\ref{dim} to find a measure $\mu_1$ satisfying a power law with exponent $\beta_1 \geq \beta-\frac{\varepsilon}{2}$ with $\supp(\mu_1)\subseteq\C\cap\BA\left(i_1,j_1\right)$. Generally, given $1<n\in\N$ and a measure $\mu_n$ satisfying $\supp\left(\mu_n\right)\subseteq\bigcap_{t=1}^{n-1}\supp\left(\mu_t\right)\cap\C\cap\BA\left(i_n,j_n\right)$, use Theorem~\ref{dim} for $t=n+1$ and $\bigcap_{t=1}^n\supp\left(\mu_t\right)\cap\C$, to find a measure $\mu_{n+1}$ with $\supp\left(\mu_{n+1}\right)\subseteq\bigcap_{t=1}^{n}\supp\left(\mu_t\right)\cap\C\cap\BA\left(i_{n+1},j_{n+1}\right)$ satisfying a power law with exponent $\beta_{n+1} \geq \beta_n - \frac{\varepsilon}{2^n}$.
Note that for any $n\in\N$,
$$\supp\left(\mu_n\right) = \bigcap_{t=1}^n \supp\left(\mu_t\right)\subseteq\bigcap_{t=1}^n\BA\left(i_t,j_t\right),$$
so in particular, by compactness of $\Theta$,
$$\bigcap_{t=1}^n \supp\left(\mu_t\right)\neq\varnothing \;\; \Rightarrow \;\;\bigcap_{t=1}^\infty \supp\left(\mu_t\right)\neq\varnothing.$$
\end{proof}

%\ref{main_thm}
\section{Proof Of Theorem 8}\label{mainKey}
Following Badziahin-Pollington-Velani, define
$$C(n,\ell,k) = \{L\in C(n,\ell):2^kR^{n-1}\leq H(A,B)<2^{k+1}R^{n-1},\;\;n,\ell,k\in\NEFES\}.$$
Then by (\ref{Cn}) and (\ref{Cnl}) we have
$$C(n,\ell) = \bigcup_{k=0}^{\lceil\log_2 R\rceil-1} C(n,\ell,k).$$
To prove Theorem~\ref{main_thm}, it'll be enough to prove
\begin{thm}\label{thm4}
Let $n,\ell,k\geq0$, and $J\in\j_{n-\ell}$. For $\varepsilon,R$ that satisfy
\begin{equation}\label{epsilonR3}
R^{-\varepsilon} + R^{\varepsilon-\alpha\beta} < \frac{1}{2}\left(\frac{b_1}{4b_2}\right)^2
\end{equation}
\begin{equation}\label{epsilonR4}
R^{\alpha\beta - \left(\frac{4}{\beta ij}+1\right)\varepsilon} > c_5
\end{equation}
\noindent
where
\begin{equation}\label{c5}
c_5 = 4^{\frac{2}{ij}+2}\frac{b_2}{b_1},
\end{equation}
we have
$$\#\{I\in\i_{n+1}(J):\exists L\in C(n,\ell,k)\;I\cap\dl\neq\varnothing\}\leq R^{\beta-\varepsilon}.$$
\end{thm}
\begin{proof}[Deduction of Theorem~\ref{main_thm} from Theorem~\ref{thm4}]
\indent
Let $\varepsilon_0$ be as in (\ref{epsilon0}) and
$$\varepsilon_1=2\varepsilon_0=\frac{\alpha\beta^2 ij}{10}.$$
Note that
$$\left(1+\frac{4}{\beta ij}\right)\varepsilon_1=\frac{(\beta ij+4)\alpha\beta}{10}<\frac{\alpha\beta}{2},$$
so substituting $\varepsilon=\varepsilon_1$ in the conditions of Theorem~\ref{thm4}, it is enough to ask for the simpler conditions
$$R^{\frac{\alpha\beta^2 ij}{10}} > \frac{64b_2^2}{b_1^2},$$
$$R^{\frac{\alpha\beta}{2}} > c_5,$$
Let $R\geq R_1$ where $R_1$ is as in (\ref{R1}). Evidently, these conditions are satisfied with $\varepsilon_1,R$. Therefore for every $0\leq k < \log_2R$,
$$\#\{I\in\i_{n+1}(J):\exists L\in C(n,\ell)\;I\cap\dl\neq\varnothing\}\leq R^{\beta-\varepsilon_1}.$$
Using the fact that $R\geq R_1\geq R_0$, where $R_0$ is as in (\ref{R0}), we get
$$\#\{I\in\i_{n+1}(J):\exists L\in C(n,\ell)\;I\cap\dl\neq\varnothing\}\leq R^{\beta-\varepsilon_1}\log_2R\leq R^{\beta-\varepsilon_0}.$$
\end{proof}
\noindent
The conditions (\ref{epsilonR3}), (\ref{epsilonR4}) arise naturally in the proof of Theorem~\ref{thm4}. To prove it, we cite $4$ propositions from \cite{Badziahin}. We only add a notation for convenience and state the propositions using the new notation. For the proofs see \cite{Badziahin}. For $n,\ell,k\in \NEFES$, $J\subseteq \Theta$, denote
$$C(n,\ell,k,J)=\{L\in C(n,\ell,k):L\cap J\neq\varnothing\},$$
and for any $P=\left(\frac{p}{q},\frac{r}{q}\right)$ denote
$$C(n,\ell,k,J,P)=\{L\in C(n,\ell,k,J):P\in L\}.$$
By putting the sign $\cdot$ at any coordinate (except for the first) we mean indifference with respect to that coordinate. For example,
$$C(n,\cdot,k)=\bigcup_{\ell=0}^{\frac{nj}{\lambda(j+1)}}C(n,\ell,k)$$
$$C(n,\ell,\cdot,J,P)=\{L\in C(n,\ell): L\cap J\neq\varnothing, \;\; P\in L\}.$$
\begin{prop}[cf. \cite{Badziahin}, Theorem 3]\label{thm3}
Let $n,\ell\in\NEFES$, $J$ be an interval of length $|J|\leq c_1R^{-n+\ell}$. Then there exists a rational point $P$ such that $C(n,\ell,\cdot,J)=C(n,\ell,\cdot,J,P)$.
\end{prop}
\begin{rem*}
In \cite{Badziahin}, this theorem is phrased slightly differently, because there $\alpha=\frac{ij}{4}$ while in this paper, adjusting to the setting of power law measures requires $\alpha=\frac{\beta ij}{4}$. The proof actually only uses the fact $\alpha>0$. The reason for choosing $\alpha$ in this specific way will become clear in the proof of Theorem~\ref{thm4}.
\end{rem*}
\begin{prop}[cf. \cite{Badziahin}, Section 5.2, Lemma 2]\label{5.2}
Let $n,k\in\NEFES$, $J\subseteq\Theta$, $P=\left(\frac{p}{q},\frac{r}{q}\right)$, $L_1,L_2\in C(n,\cdot,k,J,P)$, $L_1\neq L_2$.
Set $\tau = |J|R^{n}$. Then there exists $0<\delta<1$ such that
$$|q\theta-p|=\delta\frac{\tau2^{k+1+i}}{q^iR}.$$
\end{prop}
\begin{prop}[cf. \cite{Badziahin}, Section 5.3]\label{F}
Under the notations of Proposition \ref{5.2}, one of the lines satisfies
\begin{equation}\label{FEq}
(A,B)\in\F = \left\{(A,B):|A|<\left(c_2B\right)^i, \;\; 0<B<c_2^\frac{j}{i}\right\},
\end{equation}
where
\begin{equation}\label{c2}
c_2=\frac{q^i}{2^i\delta}.
\end{equation}
Moreover, if for some $\ell>0$, $L_1,L_2\in C(n,\ell,k,J,P)$ then one of the lines $L_1,L_2$ satisfies
\begin{equation}\label{FlEq}
(A,B)\in\F_\ell = \left\{(A,B):|A|<\left(c_2B\right)^i<c_3^ic_2\right\},
\end{equation}
where
\begin{equation}\label{c3}
c_3=c_3(\ell)=R^{\frac{j-\lambda \ell(j+1)}{i}}.
\end{equation}
\end{prop}
\begin{prop}[cf. \cite{Badziahin}, Section 5.5, Proposition 1]\label{5.4}
Let $n,\ell\in\NEFES$, $0\leq k < \log_2R$, $P=\left(\frac{p}{q},\frac{r}{q}\right)$, and
$$\tau\geq cR2^{-k}.$$
Then there exists a line $L_0\left(A_0,B_0,C_0\right)$ that passes through $P$ and satisfies $H\left(A_0,B_0\right)<R^n$, such that for every subinterval $G\subseteq\Theta$ of length $|G|=\tau R^{-n}$, one of the following holds:
\begin{enumerate}
    \item $\# C(n,\ell,k,G,P)\leq 1$.
    \item Every $L\in C(n,\ell,k,G,P)$ satisfies $\dl\subseteq 2\Delta(L_0)$ besides possibly $1$ exceptional line.
    \item $\delta$ from Proposition \ref{5.2} satisfies
\begin{equation}\label{delta}
\delta > c_4\left(\frac{cR}{2^k\tau}\right)^\frac{2}{j}
\end{equation}
where
\begin{equation}\label{c4}
c_4=4^{-\frac{2}{j}}2^{-i}.
\end{equation}
\end{enumerate}
\end{prop}
\begin{proof}[Proof of Theorem~\ref{thm4}]
Set $n,\ell,k\geq 0$ and $J\in\j_{n-\ell}$. We wish to show that lines from $C(n,\ell,k,J)$ remove at most $R^{\beta-\varepsilon}$ intervals $I\in\i_{n+1}(J)$.
$$|\Delta(L)|=\frac{2c}{H(A,B)}\leq2cR^{-n+1}2^{-k}=c_12^{-k+1}R^{-n-\alpha},$$
so for any $I\in\i_{n+1}(J)$
\begin{equation}\label{dlRemove}
\frac{\mu(\Delta(L))}{\mu(I)}\leq \frac{b_2\left(c_12^{-k+1}R^{-n-\alpha}\right)^\beta}{b_1\left(c_1R^{-n-1}\right)^\beta}=
\frac{b_2}{b_1}\left(R^{1-\alpha}2^{-k+1}\right)^\beta.
\end{equation}
Then
\begin{equation}
K^* = \frac{b_2}{b_1}\left(R^{1-\alpha}2^{-k+1}\right)^\beta+2
\end{equation}
is an upper bound on the number of intervals that can be removed by a line $L\in C(n,\ell,k,J)$, and it satisfies
\begin{equation}\label{K1}
K^* \leq \frac{4b_2}{b_1}K^\beta,
\end{equation}
where
\begin{equation}\label{K}
K = \left\{
          \begin{array}{ll}
          R^{1-\alpha}2^{-k} & R^{1-\alpha}2^{-k}>1 \\
          1 & R^{1-\alpha}2^{-k}\leq1.
          \end{array}
    \right.
\end{equation}
Set $d=\lceil\frac{R^{1-\frac{2\varepsilon}{\beta}}}{K}\rceil$. Then $d\geq \frac{R^{1-\frac{2\varepsilon}{\beta}}}{K}$ so
$$\frac{|J|}{d}\leq\frac{Kc_1R^{\ell-n}}{R^{1-\frac{2\varepsilon}{\beta}}}\leq\tau R^{-n},$$
where
\begin{equation}\label{tau}
\tau=\left\{
          \begin{array}{ll}
          R^{\ell-\alpha+\frac{2\varepsilon}{\beta}}2^{-k}c_1  & R^{1-\alpha}2^{-k}>1 \\
          R^{\ell-1+\frac{2\varepsilon}{\beta}}c_1 & R^{1-\alpha}2^{-k}\leq1.
          \end{array}
     \right.
\end{equation}
Note that in both cases
$$\tau\geq cR2^{-k}.$$
By Proposition~\ref{thm3}, there exists a rational point $P$ such that $C(n,\ell,k,J)=C(n,\ell,k,J,P)$. Using the one-dimensionality of $J$, there exists a covering $\{G_i\}_{i=1}^{d^*}$ of $J\cap\C$ by intervals of length $\frac{|J|}{d}$ centered in $\C$, such that every $x\in J\cap\C$ is contained in at most two $G_i$'s. Since $\C=\supp(\mu)$ and $\mu$ satisfies a power law, $d^*$ must satisfy
$$d^* \leq \frac{2\mu(J\cap\C)}{\min_{1\leq i\leq d^*}G_i}+2 \leq 4\frac{b_2}{b_1}d^\beta.$$
Consider $C\left(n,\ell,k,G_i,P\right)$. Note that $|G_i|\leq\tau R^{-n}$, and that by definition of $K^*$, for each line $L$, $\dl$ intersects at most $K^*$ intervals from $\i_{n+1}(J)$. Therefore, if for every $1\leq i\leq {d^*}$, $C\left(n,\ell,k,G_i\right)$ consists of only $1$ line then by (\ref{K1}), they all remove at most
\begin{equation}\label{dK}
d^*K^* \leq 2\left(\frac{4b_2}{b_1}\right)^2R^{\beta-2\varepsilon}.
\end{equation}
We shall show that either we have to deal with only one more line, or otherwise we will have a useful bound on the number of lines in $C(n,\ell,k,J,P)$.
\textbf{Case 1, }$\mathbf{\boldsymbol\delta\leq c_4\left(\frac{cR}{2^k\boldsymbol\tau}\right)^{\frac{2}{j}}}$\textbf{.}  Viewing Proposition~\ref{5.4}, for each $C\left(n,\ell,k,G_i,P\right)$ there are at most two relevant lines, one exceptional line in each $C\left(n,\ell,k,G_i,P\right)$ and one line $L_0$ with $H\left(A_0,B_0\right)<R^n$ which is the same for every $i$ with $\# C\left(n,\ell,k,G_i,P\right)>1$.
If $L_0\in C\left(n_0\right)$ for some $n_0<n$, then intervals that intersect $\Delta\left(L_0\right)$ were obviously removed during the $\left(n_0+1\right)$'th step. Moreover, if there were some $J_1\in\j_{n_0+1},J_2\in\j_{n_0+2}\left(J_1\right)$ such that $J_2\cap 2\Delta\left(L_0\right)\neq\varnothing$ then $J_1^-\cap 2\Delta\left(L_0\right)\neq\varnothing$ and by (\ref{tmp}), $J_1\cap\Delta\left(L_0\right)\neq\varnothing$, but then $J_1$ was already removed in the $\left(n_0+1\right)$'th step. Thus $2\Delta\left(L_0\right)$ cannot remove any interval from $\j_{n_0+2}$, and since $n_0<n$, neither from $\j_{n+1}$.
If $L_0\in C(n)$ then by the same calculation as in (\ref{dlRemove}), $2\Delta\left(L_0\right)$ may remove at most $$\frac{b_2}{b_1}\left(4R^{1-\alpha}\right)^\beta+2$$
intervals.
Finally, in this case where $\delta\leq c_4\left(\frac{cR}{2^k\tau}\right)^{\frac{2}{j}}$, using (\ref{dK}) we get that there are at most
$$2\left(\frac{4b_2}{b_1}\right)^2R^{\beta-2\varepsilon} + \frac{8b_2}{b_1}R^{\beta(1-\alpha)}$$
subintervals $I\in\i_{n+1}(J)$ to be removed. Using (\ref{epsilonR3}) we get the estimation we wanted.
\newline
\noindent
\textbf{Case 2, }$\mathbf{\boldsymbol\delta > c_4\left(\frac{cR}{2^k\boldsymbol\tau}\right)^{\frac{2}{j}}}$\textbf{.}  Denote the number of lines in $C(n,\ell,k,J,P)$ by $M$. By Proposition~\ref{F},
$$M^* = \left\{
              \begin{array}{ll}
              \#\{L\in C(n,\ell,k,J,P):(A,B)\in\F\} & \ell=0 \\
              \#\{L\in C(n,\ell,k,J,P):(A,B)\in\F_\ell\} & \ell>0
        \end{array}
        \right.$$
satisfies $M\leq M^*+1.$ No two points $\left(A_1,B_1\right),\left(A_2,B_2\right)$ are on the same line through the origin, because if they were then the lines $L_1\left(A_1,B_1,C_1\right)$ and $L_2\left(A_2,B_2,C_2\right)$ would be parallel, contradicting that they intersect in $P$. It follows that these points create disjoint triangles with the origin $(0,0)$. Each triangle has area at least $\frac{q}{2}$, and the area of the union of triangles can't exceed the area of $\F$. By Definition (\ref{c2}) of $c_2$, $c_2=\frac{q^i}{2^i\delta}$, so by (\ref{FEq})
$$|\F| \leq 2c_2^{\frac{1}{i}} = q\delta^{-\frac{1}{i}},$$
For $\F_\ell$, $\ell>0$, by (\ref{FlEq}) and (\ref{c3}),
$$|\F_\ell| \leq 2c_2^{\frac{1}{i}}c_3^{1+i} = R^{\frac{\left(j-\lambda \ell(j+1)\right)(i+1)}{i}}q\delta^{-\frac{1}{i}}.$$
To ease calculations, use (\ref{ij}) and (\ref{lambda}) to write
$$\frac{\left(j-\lambda \ell(j+1)\right)(i+1)}{i}=\frac{j-i^2j-6\ell}{ij}-3\ell\leq-\frac{5\ell}{ij}.$$
Thus for any $\ell\geq 0$
\begin{equation}\label{M}
M\leq 2\delta^{-\frac{1}{i}}R^{-\frac{5\ell}{ij}}+2.
\end{equation}
We will show that $MK^*\leq R^{\beta-\varepsilon}$, and we are done with the proof of Theorem~\ref{thm4}.
Using (\ref{delta}) we have
\begin{equation}\label{deltai}
\delta^{-\frac{1}{i}}<c_4^{-\frac{1}{i}}\left(\frac{cR}{2^k\tau}\right)^{-\frac{2}{ji}}.
\end{equation}
By (\ref{tau})
\begin{center}
\begin{equation}\label{cRtau}
\frac{cR}{2^k\tau}\geq\left\{
                            \begin{array}{ll}
                            R^{-\ell-\frac{2\varepsilon}{\beta}}  & R^{1-\alpha}2^{-k}>1\\
                            R^{-\ell-\alpha-\frac{2\varepsilon}{\beta}} & R^{1-\alpha}2^{-k}\leq1.
                            \end{array}
                      \right.
\end{equation}
\end{center}
\textbf{Case 2.1, }$\mathbf{R^{1-\boldsymbol\alpha}2^{-k}>1}$\textbf{.}  By (\ref{M}), (\ref{deltai}), (\ref{cRtau}) and (\ref{c4})
\begin{equation}\label{Mlast}
M<2\cdot 4^{\frac{2}{ij}}\left(R^{\frac{4\varepsilon}{\beta}-3\ell}\right)^{\frac{1}{ij}}+2<
4^{\frac{2}{ij}+1}R^{\frac{4\varepsilon}{\beta ij}}.
\end{equation}
Using (\ref{K}) and $2^{-k}<1$,
\begin{equation}\label{Klast}
K^*\leq \frac{4b_2}{b_1}R^{\beta(1-\boldsymbol\alpha)}.
\end{equation}
Combine (\ref{Mlast}), (\ref{Klast}) and (\ref{c5}) to get,
$$MK^* < c_5R^{\beta - \alpha\beta + \frac{4\varepsilon}{\beta ij}}.$$
By (\ref{epsilonR4}),
$$MK^* < R^{\beta-\varepsilon}.$$
\newline
\noindent
\textbf{Case 2.2, }$\mathbf{R^{1-\boldsymbol\alpha}2^{-k}\leq1}$\textbf{.}  By (\ref{M}), (\ref{deltai}), (\ref{cRtau}) and (\ref{c4})
\begin{equation}\label{Mlast2}
M<2\cdot 4^{\frac{2}{ij}}\left(R^{\frac{4\varepsilon}{\beta}+2\alpha-3\ell}\right)^{\frac{1}{ij}}+2<
4^{\frac{2}{ij}+1}R^{\frac{\beta}{2} + \frac{4\varepsilon}{\beta ij}}.
\end{equation}
and by (\ref{K})
\begin{equation}\label{Klast2}
K^*\leq \frac{4b_2}{b_1}.
\end{equation}
Combine (\ref{Mlast2}), (\ref{Klast2}) and (\ref{c5}) to get,
$$MK^* < c_5R^{\frac{\beta}{2}+\frac{4\varepsilon}{\beta ij}}.$$
Note that because of (\ref{alpha}), $\frac{\beta}{2}+\frac{4\varepsilon}{\beta ij} < \beta - \beta\alpha + \frac{4\varepsilon}{\beta ij}$ so we are done.
\end{proof}

%\section{A More General Construction}\label{generalConstruction}
%\input{generalConstruction}

%\begin{appendices}
\appendix
\gdef\thesection{Appendix \Alph{section}}

\section{Measure On The Limit Set Of A Tree-Like Family}\label{treeLikeAppendix}
\begin{proof}[proof of Lemma~\ref{treeDim}]
We remark that $\gamma R\in\N$. Assume first that $n_0=0$, $\t_0=\{J_0\}$, $|J_0|=1$. For every $n\in\NEFES$ define $\nu_n$ by distributing it equally on each element of $\t_n$, i.e.,
$$\nu_n = \frac{\sum_{I\in\t_n}\l|_I}{(\gamma R)^n},$$
where $\l|_I$ is the restriction of the Lebesgue measure to the interval $I$, i.e., for any $A\subseteq \R$, $\l|_I(A)=\frac{\l(A\cap I)}{\l(I)}$.
$\nu_n$ is a probability measure because of (\ref{numOfChildren}). Thus, there is a weak-* convergent subsequence $\{\nu_{n_k}\}_{k\in\N}$. Denote its limit by $\nu$. Then,
$$\supp(\nu)=\bigcap_{k\in\N}\bigcup_{I\in\t_{n_k}}I.$$
We have $\forall I\in\t_{n+1}\;\exists J\in\t_n\;\;I\subseteq J$ so actually
\begin{equation}\label{treeBigcap}
\supp(\nu)=\bigcap_{n\in\N}\bigcup_{I\in\t_n}I.
\end{equation}
Also, for every $n\in\N$, $I\in\t_n$ and every $m\geq n$,
$\nu_m(I)=\nu_n(I)=(\gamma R)^{-n}=\left(R^{-n}\right)^\beta$ and thus
\begin{equation}\label{treeMeas}
\nu(I)=\left(R^{-n}\right)^\beta.
\end{equation}
Let $B(x,r)$ be any ball of radius $r$ and center $x\in\supp(\nu)$, and let $n$ be such that
$$R^{-n-1}\leq r\leq R^{-n}.$$
For the left hand inequality in Definition (\ref{powerLawEq}), $x\in\supp(\nu)$ so by (\ref{treeBigcap}) there exists $I\in\t_{n+1}$ such that $x\in I$, therefore $I\subseteq B(x,r)$, so by (\ref{treeMeas})
$$\nu(B(x,r))\geq \left(R^{-n-1}\right)^\beta\geq\frac{1}{R^\beta}r^\beta.$$
For the right hand inequality in Definition (\ref{powerLawEq}),
$$\# \{I\in\t_n:I\cap B(x,r)\neq\varnothing\}\leq 3\;\;\Rightarrow\;\;\nu(B(x,r))\leq 3\left(R^{-n}\right)^\beta,$$
so $\nu(B(x,r))\leq 3R^\beta r^\beta$.
Finally $\nu$ satisfies the definition of power law (\ref{powerLawEq}) with $b_1=\frac{1}{R^\beta}$ and $b_2 = 3R^\beta$.
In the general case where $n_0\neq0$, we start the construction from $n\geq n_0$, and again define $\nu_n$ by distributing equally the Lebesgue measure of each element in $\t_{n_0}$
$$\nu_n = \frac{\sum_{I\in\t_n}a(I)\l|_I}{A(\gamma R)^{n}}.$$
where $a(I)=|J|$ for the unique $J\in\t_{n_0}$ such that $I\subseteq J$, and $$A = (\gamma R)^{-n_0}\sum_{J\in\t_{n_0}}|J|.$$
Define $\nu$ as above. (\ref{treeBigcap}) is satisfied, and instead of (\ref{treeMeas}) we have
\begin{equation}\label{treeMeasGen}
\nu(I)=\frac{a(I)}{A}\left(R^{-n}\right)^\beta.
\end{equation}
Let $B(x,r)$ be any ball of radius $r$ and center $x\in\supp(\nu)$, and let $n$ be such that
$$R^{-n-1}\leq r\leq R^{-n}.$$
On one hand, $x\in\supp(\nu)$ so by (\ref{treeBigcap}) there exists $J\in\t_{n+1}$ such that $x\in J$, therefore $J\subseteq B(x,r)$, so by (\ref{treeMeasGen})
$$\nu(B(x,r))\geq \frac{a(J)}{A}\left(R^{-n-1}\right)^\beta\geq \frac{a(J)}{A}\frac{1}{R^\beta}r^\beta.$$
On the other hand,
$$\# \{J\in\t_n:J\cap B(x,r)\neq\varnothing\}\leq 3\;\;\Rightarrow\;\;\nu(B(x,r))\leq 3\frac{\max_{J\in\t_{n_0}}|J|}{A}\left(R^{-n}\right)^\beta,$$
so $\nu(B(x,r))\leq 3\frac{\max_{J\in\t_{n_0}}|J|}{A}R^\beta r^\beta$.
Finally $\nu$ satisfies the definition of power law (\ref{powerLawEq}) with $b_1=\frac{\min_{J\in\t_{n_0}}|J|}{A}\frac{1}{R^\beta}$ and $b_2 = 3\frac{\max_{J\in\t_{n_0}}|J|}{A}R^\beta$.
\end{proof}
%\end{appendices}

\section{$\BAij$ Is Absolutely Winning On $\C$ (joint with Barak Weiss)}\label{abswin}
The work described in the body of this paper was done prior to the appearance of Jinpeng An's work \cite{JinpengAn} on Arxiv. In this appendix we explain how An's work can be used to obtain a strengthening of the results of this paper. In particular, we prove a result about the Hausdorff dimension.
\begin{thm}\label{iOurThm}
Let $\C\subseteq\Theta$ be the support of a measure satisfying a power law, and let $\{\left(i_t,j_t\right)\}_{t\in \N}$ with $\left(i_t,j_t\right)$ as in (\ref{ij}). Then
$$\dim\left(\mathbf{C}\cap\bigcap_{t\in\N} \BA\left(i_t,j_t\right)\right)=\dim(\mathbf{C}).$$
\end{thm}
\begin{rem*}
Under the weaker assumption that $\mu$ is $\gamma$ absolutely decaying (see \cite{BFKRW}, \S5 for the definition) the same argument gives the conclusion $$\dim\left(\C\cap\bigcap_{t\in\N} \BA\left(i_t,j_t\right)\right)\geq\gamma.$$
\end{rem*}
\noindent
To prove Theorem~\ref{iOurThm}, we use the notion of an \emph{absolute winning} set, as defined by McMullen in \cite{Mc} and generalized to the notion of a \emph{hyperplane absolute winning} (HAW) in \cite{BFKRW}. Let $X\subseteq\R$ and let $\beta>0$. The $\beta$-absolute game is defined as follows. Bob starts by choosing a closed ball $B_0 = B\left(x_0,r_0\right)$ with $x_0\in X$ and $r_0>0$. The game continues in the $n$'th step, $n\geq1$, with Alice choosing a $\beta_n$-neighborhood $A_n$ of a point in $\R$, where $\beta_n \leq \beta r_{n-1}$, and Bob choosing a closed ball
$$B_n = B\left(x_n,r_n\right) \subseteq B_{n-1} \setminus A_n,$$
with $x_n\in X$ and $r_n \geq \beta r_{n-1}$. A set $S\subseteq X$ is \emph{$\beta$-absolute winning on $X$} if Alice can force $\bigcap_{n=0}^\infty B_n\cap S \neq \emptyset$. One advantage of the absolute winning property is that it passes to certain subsets:
\begin{defn}(cf. \cite{BFKRW}, Definition 4.2)\label{hyperplanediffuse}
A closed set $K \subseteq \R$ is said to be \emph{$\beta$-diffuse}, $0 < \beta < 1$, if there exists $\rho_K>0$ such that for any $0 < \rho < \rho_K$, $x\in K$ and $x'\in\R$
$$\left(K\cap B(x, \rho)\right)\setminus B(x',\beta\rho)\neq\varnothing.$$
We say that $K$ is \emph{diffuse} if it is $\beta$-diffuse for some $0 < \beta < 1$.
\end{defn}
\noindent
For diffuse sets we define
$$\beta_0(K)=\sup\left\{\frac{\beta}{\beta+2}:K \; {\rm is } \; \beta{\rm-diffuse}\right\}.$$
It is clear that $\beta_0(\R)=\frac{1}{3}$. Also, if $\beta>\beta_0(K)$, it is possible that Bob will not have an available move to make, and our game is ill-defined. We will consider the absolute game played on a diffuse set $K$, where Bob first chooses a $0<\beta<\beta_0(K)$ and the game continues as a $\beta$-absolute game on $K$, and say that $S$ is absolute winning on $K$ if it is $\beta$-absolute winning on $K$ for every $0<\beta<\beta_0(K)$. It is easy to see that this is equivalent to requiring that for any $\varepsilon>0$ there is $0<\beta<\min\{\varepsilon, \beta_0(K)\}$ such that $S$ is $\beta$-absolute winning on $K$.
\begin{prop}[\cite{BFKRW}, Proposition 4.9]\label{abs}
Assume $S\subseteq \R$ is absolute winning on $\R$ and let $K\subseteq\R$ be diffuse. Then $S\cap K$ is absolute winning on $K$.
\end{prop}
\noindent
As an example of a diffuse set one can take the support of a measure satisfying a power law. Two additional advantages of using games, and in particular the absolute game, are the infinite intersection and the full Hausdorff dimension properties.
\begin{prop}[\cite{Mc} page 3, or \cite{BFKRW} Proposition 2.3(b)]\label{infiniteIntersection}
For every $n\in\N$, assume $S_n\subseteq R$ is absolute winning on $\R$. Then $\bigcap_{n\in\N}S_n$ is absolute winning on $\R$.
\end{prop}
\noindent
To get the full Hausdorff dimension of the intersection with nice fractals, we note that being absolute winning implies being winning in the original sense due to Schmidt \cite{S}. Specifically:
\begin{prop}[\cite{BFKRW}, Proposition 4.7]\label{absSchmidt}
Let $K\subseteq\R$ be diffuse and assume $S\subseteq K$ is absolute winning on $K$. Then $S$ is winning on $K$.
\end{prop}
\begin{prop}[\cite{Fishman}, Theorem 5.1]\label{fullDimension}
Assume $K\subseteq\R$ is the support of a measure satisfying a power law, and $S\subseteq K$ is winning on $K$. Then, dim(S)=dim(K).
\end{prop}
\noindent
We will need a variant of the absolute game.
\begin{defn}\label{multipleDefn}
Fix an integer $N\in\N$ and change only the following: in every step $n\geq 1$ allow $A_n$ to be the union of up to $N$ neighborhoods of points in $\R$ of radius not bigger than $\beta r_{n-1}$. Call this game $(N,\beta)$-absolute game. A set $S\subseteq K$ which is winning for this game played on $K$ will be called \emph{$(N,\beta)$-absolute winning on $K$}.
\end{defn}
\begin{defn}\label{Nhyperplanediffuse}
A closed set $K \subseteq \R$ is said to be \emph{$(N,\beta)$-diffuse}, $0 < \beta < 1$, if there exists $\rho_K>0$ such that for any $0 < \rho < \rho_K$, $x\in K$ and $x_1,...,x_N\in R$
$$\left(K\cap B(x, \rho)\right) \setminus \bigcup_{k=1}^NB\left(x_k,(\beta\rho)\right)\neq\varnothing.$$
We say that $K$ is \emph{$N$-diffuse} if it is $(N,\beta)$-diffuse for some $0 < \beta < 1$ (since $N\in\N$ and $\beta<1$ there is no ambiguity in this notation).
\end{defn}
\noindent
For $N$-diffuse sets we define
$$\beta_0(K,N)=\sup\left\{\frac{\beta}{\beta+2}:K \; {\rm is } \; (N,\beta){\rm-diffuse}\right\}.$$
As before, we will consider the $N$-absolute game played on a $N$-diffuse set $K$, where Bob first chooses a $0<\beta<\beta_0(K,N)$ and the game continues as a $(N,\beta)$-absolute game, and say that $S$ is $N$-absolute winning on $K$ if it is $(N,\beta)$-absolute winning on $K$ for every $0<\beta<\beta_0(K,N)$. It is left to the reader to see that
\begin{lem}\label{NdiffuseLem}
If $K\subseteq \mathbb{R}$ is diffuse then for every $N\in\N$, $K$ is $N$-diffuse.
\end{lem}
\begin{lem}\label{multiple}
A set $S\subseteq X$ is $N$-absolute winning on $X$ if and only if $S$ is absolute winning on $X$.
\end{lem}
\begin{proof}
Note that in Definition~\ref{multipleDefn}, Alice may also use less than $N$ neighborhoods. So a set which is absolute winning on $X$ is obviously $N$-absolute winning on $X$. Assume $S$ is $N$-absolute winning on $X$ and define a strategy for Alice. Let $0 < \beta < \beta_0$. Then $0 < \beta^N<\beta_0$, so there is a winning strategy for Alice in the $\beta^N$ $N$-absolute game. Let $B_n$ be the $n$'th ball Bob chose in the $\beta$-absolute game. Then, $\{B_{nN}\}_{n=0}^\infty$ is a legitimate sequence of balls in the $\beta^N$ $N$-absolute game. Let $\bigcup_{i=1}^NA_n(i)$ be the $n$'th choice of Alice using her winning strategy. Then, for every $n\in\N$ write $n=qN+r$ with $1\leq r\leq N$ and $q\in\NEFES$, and let Alice choose $A_n=A_q(r)$. We have,
$$\bigcap_{n=0}^\infty B_n\cap S = \bigcap_{n=0}^\infty B_{nN}\cap S \neq \emptyset.$$
So $S$ is winning for the $\beta$-absolute game on $X$.
\end{proof}
\noindent
Now we're going to use this Lemma in order to show that the arguments of \cite{JinpengAn} imply that $\BAij\cap\Theta$ is not only winning but is absolute winning.
\begin{thm}[cf. Jinpeng An \cite{JinpengAn}, Proposition 3.1]\label{jinpengLemma}
For any $R>8$, a closed interval $B\subseteq\Theta$ and a $\lfloor R\rfloor$-regular tree-like family $\t = \{\t_n\}_{n\in\NEFES}$ such that $\t_0=\{B\}$ and for every $I\in\t_n$, $|I|=|B|R^{-n}$, there exists a $(\lfloor R\rfloor-5)$-regular tree-like subfamily $\i$ such that
$$\bigcap_{n=0}^\infty\bigcup_{I\in\i_n}I\subseteq\BAij\cap\Theta.$$
\end{thm}
\begin{prop}\label{BAisHAW}
$\BAij\cap\Theta$ is absolute winning on $\R$.
\end{prop}
\begin{proof}
Let Bob choose the ball $B_0=B\left(x_0,r_0\right)\subseteq\Theta$, and $0<\beta<\frac{1}{3}$.
Define $R=\frac{1}{\beta^2}$. Let $\t$ be the tree-like family of closed intervals that is generated by the recursive procedure of taking $\lfloor R\rfloor$ subintervals of length $\frac{1}{R}$ from the previous level, starting from the left side of each interval. Since $\beta<\frac{1}{3}$, $R>8$ and by Proposition~\ref{jinpengLemma} there exists a $\lfloor R\rfloor-5$ regular subtree $\i$. We use it to define a winning strategy for Alice for the $N$-absolute game with $N=12$. On her first turn, Alice chooses
\begin{equation}\label{base1}
A_1 = \bigcup_{I\in\t_1 \setminus\i_1}I\cup\left[x_0 - r_0 + 2\frac{\lfloor R\rfloor}{R}r_0,x_0 + r_0\right],
\end{equation}
which is a union of at most $6$ intervals. Note that by the definition of $\i$,
$$\left(B_0\setminus A_1\right)\cap\bigcup_{I\in\t_n\setminus\i_n}I=\varnothing.$$
In the following moves of Alice plays dummy moves by choosing the empty set, except for the turns $s_n$ in which Bob chooses for the first time a ball of radius $r$ that satisfies
\begin{equation}\label{rjn}
\frac{\beta r_0}{R^n} \leq r\leq \frac{r_0}{R^n}
\end{equation}
(If this doesn't happen, Alice continues playing dummy moves and wins because $\BAij\cap\Theta$ is dense). Assume that Alice chose $A_{s_{n-1}}$ such that $\left(B_{s_{n-1}-1}\setminus A_{s_{n-1}}\right)\cap\left(\bigcup_{I\in\t_n\setminus\i_n}I\right)=\varnothing$. This is true for $n=1$ by (\ref{base1}), where $s_0$ is defined to be $1$. By the RHS of (\ref{rjn}) there exist $I_1, I_2\in\t_n$ such that $B_{s_n}\subseteq I_1\cup I_2$. By the induction hypothesis $I_1,\;I_2$ are actually in $\i_n\subseteq\t_n$. By the construction of $\i$, both $I_1, I_2$ contain at most $5$ intervals that are not in $\i_{n+1}$. Taking into account also the rightmost subinterval of each of them, Alice chooses $A_{s_n+1}$ to be a union of at most $12$ intervals. Note that by the LHS of (\ref{rjn}) any $I\in\i_{n+1}$ satisfies
$$|I|=\frac{2r_0}{R^{n+1}}=\frac{2\beta^2 r_0}{R^{n}}\leq\beta|B|,$$
so Alice can indeed do it by the rules of our game. We still have to show that by doing so Alice does not lose the game by leaving Bob with no possible continuation. For that we show that there is a ball $B$ of radius $r\geq\frac{r_0}{R^{n+1}}$ such that $B\subseteq B_{{s_n}-1}\setminus A_{s_n}$. It is sufficient to show that $|B_{{s_n}-1}\setminus \tilde{A}_{s_n}|>0$, where $\tilde{A}_{s_n}$ is a $\frac{r_0}{R^n}$-neighborhood of $A_{s_n}$. Indeed, for $n=1$ there is nothing to prove since $R>8$ and Alice removed at most $6$ subintervals. For $n>1$ use (\ref{rjn}) to get
$$|B_{s_n}\setminus \tilde{A}_{s_n}|\geq2\left(\frac{\beta r_0}{R^n}-24\frac{r_0}{R^{n+1}}\right)=\frac{2r_0}{R^{n+1}}\left(\frac{1}{\beta}-24\right).$$
In case $\beta<\frac{1}{24}$ we are done. If $\beta\geq\frac{1}{24}$ we can set $R=\frac{1}{\beta^4}$ using the same reasoning with $\beta^3$. Since $\beta^3<\left(\frac{1}{3}\right)^3<\frac{1}{24}$ we will be done.
This defines a winning strategy for Alice in the absolute game with $N=12$, because
$$\bigcap_{n=0}^\infty B_n=\bigcap_{n=1}^\infty B_{s_n}\in\bigcap_{n=1}^\infty\bigcup_{I\in\i_n}I\subseteq\BAij\cap\Theta.$$
Therefore applying Lemma~\ref{multiple} we have proved that $\BAij\cap\Theta$ is absolute winning on $\R$.
\end{proof}
\begin{proof}[Proof of Theorem~\ref{iOurThm}]
For every $t\in\N$, $\BA\left(i_t,j_t\right)\cap\Theta$ is absolute winning on $\R$. By using the infinite intersection property of absolute winning sets Proposition~\ref{infiniteIntersection} we get that $\bigcap_{t\in\N}\BA\left(i_t,j_t\right)\cap\Theta$ is absolute winning on $\R$. Therefore by Proposition~\ref{abs},  $\bigcap_{t\in\N}\BA\left(i_t,j_t\right)\cap \C$ is absolute winning on $\C$. At last, by Proposition~\ref{absSchmidt}, $\bigcap_{t\in\N}\BA\left(i_t,j_t\right)\cap \C$ is winning on $\C$ and hence by the full dimension property Proposition~\ref{fullDimension} and the fact that $\C$ is the support of a measure satisfying a power law, we get
$$\dim\left(\bigcap_{t\in\N}\BA\left(i_t,j_t\right)\cap \C\right)=\dim(\C).$$
\end{proof}

\end{document}